\documentclass[11pt]{article}
\usepackage{latexsym,amsmath,stackrel,color,amsthm,amssymb,epsfig,graphicx,mathrsfs}
\usepackage{graphicx}
\usepackage{amssymb}
\usepackage[left=1in,top=1in,right=1in,bottom=1in]{geometry}
\usepackage[linktocpage=true]{hyperref}
\usepackage{setspace}
\usepackage{amssymb, amsmath, amsthm, graphicx,mathrsfs}
\usepackage{caption}
\usepackage{titlesec}

\usepackage{comment,caption}
\usepackage{tikz}
\makeatletter
\def\thm@space@setup{%
  \thm@preskip=\parskip \thm@postskip=0pt
}
\makeatother

\usepackage{thmtools}

\declaretheoremstyle[%
  spaceabove=6pt,%
  spacebelow=6pt,%
  headfont=\normalfont\itshape,%
  postheadspace=1em,%
  qed=\qedsymbol%
]{mystyle}

\def\qed{\hfill\ifhmode\unskip\nobreak\fi\quad\ifmmode\Box\else\hfill$\Box$\fi}
\def\ite#1{\hfill\break${}$\hbox to 50pt {\quad(#1)\hfill}}
\newtheorem{thm}{Theorem}
\newtheorem{cor}{Corollary}

\newtheorem{problem}{Problem}

\newtheorem{lemma}[thm]{Lemma}
\newtheorem{conj}{Conjecture}
\newtheorem{prop}{Proposition}

\newtheorem{defn}{Definition}

\tikzstyle{vertex}=[circle,fill=black,inner sep=2pt]
\tikzstyle{vertrect}=[draw,rectangle,inner sep=2pt]
\tikzstyle{vertdia}=[draw,diamond,inner sep=2pt]

\newcommand{\N}{\mathbb{N}}

\newcommand{\A}{\mathcal{A}}
\newcommand{\F}{\mathcal{F}}
\newcommand{\B}{\mathcal{B}}

\titleformat{\subsection}[runin]
        {\normalfont\bfseries}
        {\thesubsection}
        {0.4em}
        {}
        [.]

\begin{document}
\title{A note on $k$-wise oddtown problems} 
\author{
	Jason O'Neill
	\footnote{Research supported by NSF award DMS-1800332} \\
	Department of Mathematics, \\
	University of California, San Diego \\
	\texttt{jmoneill@ucsd.edu}
	\and
	Jacques Verstra\"{e}te
\footnotemark[1] \\
	Department of Mathematics, \\
	University of California, San Diego \\
	\texttt{jverstra@math.ucsd.edu}
} 
\maketitle
\begin{abstract}
For integers $2 \leq t \leq k$, we consider a collection of $k$ set families $\mathcal{A}_j: 1 \leq j \leq k$ where $\mathcal{A}_j = \{ A_{j,i} \subseteq [n] : 1 \leq i \leq m \}$ and $|A_{1, i_1} \cap \cdots \cap A_{k,i_k}|$ is even if and only if at least $t$ of the $i_j$ are distinct. In this paper, we prove that $m =O(n^{ 1/ \lfloor k/2 \rfloor})$ when $t=k$ and $m = O( n^{1/(t-1)})$ when $2t-2 \leq k$ and prove that both of these bounds are best possible. Specializing to the case where $\A = \A_1 = \cdots = \A_k$, we recover a variation of the classical oddtown problem. 
\end{abstract}

\section{Introduction}
Given a collection $\A$ of subsets of an $n$ element set, $\A$ follows \textit{oddtown rules} if the sizes of all sets in $\A$ are odd and distinct pairs of sets from $\A$ have even sized intersections. Berlekemp \cite{BERK} proved that the size of a family which satisfies oddtown rules is at most $n$ which is easily seen to be best possible. One generalization of this classical result is the \textit{skew oddtown theorem} which appears in Babai and Frankl \cite{BF}: 

\begin{thm}[Babai, Frankl \textup{\cite{BF}}]\label{thm:skewoddtown}
Let $\mathcal{A} = \{A_1,A_2,\dots,A_m\}$ and $\mathcal{B} = \{B_1,B_2,\dots,B_m\}$ be families of subsets of an $n$ element set where $|A_i \cap B_i|$ is odd for all $i \in [m]$ and $|A_i \cap B_j|$ is even for all $i<j \in [m]$. Then $m \leq n$. 
\end{thm}

In this paper, we give analogs of Theorem \ref{thm:skewoddtown} for $k\geq 3$ families. Our main result is as follows:

\begin{thm}\label{thm:main1}
Let $(\mathcal{A}_1,\mathcal{A}_2,\dots,\mathcal{A}_k)$ be set families of an $n$ element set with $\mathcal{A}_{j} = \{A_{j,i} : 1 \leq i \leq m\}$ where
$|\bigcap_{j = 1}^k A_{j,i_j}|$ is even if and only if $i_1, \ldots, i_k \in [m]$ are all distinct. Then $m = O(n^{ 1/ \lfloor k/2 \rfloor})$.
\end{thm} 

By employing a connection to a hypergraph covering problem and a hypergraph extension (see Alon \cite{ALON2}, Cioab\u{a}-K{\"u}ndgen-Verstra\"{e}te \cite{CKV}, and Leader-Mili\'{c}evi\'{c}-Tan \cite{LMT}) of the Graham-Pollak Theorem \cite{GP}, we also show that Theorem \ref{thm:main1} is best possible. Moreover, we prove the following:

\begin{thm}\label{thm:main2}
Let $t,k$ be integers with $t \geq 2$ and $2t-2 \leq k$. If $(\mathcal{A}_1,\mathcal{A}_2,\dots,\mathcal{A}_k)$ are set families of an $n$ element set with $\mathcal{A}_{j} = \{A_{j,i} : 1 \leq i \leq m\}$ where $|\bigcap_{j = 1}^k A_{j,i_j}|$ is even if and only if at least $t$ of the $i_j$ are distinct, then $m = O(n^{1/(t-1)})$. 
\end{thm} 

Theorem \ref{thm:main2} is seen to be best possible by again using a connection to a hypergraph covering problem. In the case where $\A = \A_1 = \cdots = \A_k$, we recover a variation of the classical oddtown problem. For integers $t< k$, a  family $\A \subseteq 2^{[n]}$ is said to follow \textit{$(k,t)$-oddtown rules} if the intersection of any $d$ distinct sets in $\A$ is odd if $d<t$ and even if $t \leq d \leq k$. 

\begin{cor}\label{cor:oddtownvariation}
Let $t,k \in \N$ be so that $2 \leq t \leq k$ and $2t-2 \leq k$. If $\A \subset 2^{[n]}$ follows $(k,t)$-oddtown rules, then $|\A| = O( n^{1/(t-1)})$. 
\end{cor}

Corollary \ref{cor:oddtownvariation} is best possible by considering $ \A_i = \{ A \in \binom{[n]}{t-1} : i \in A \} $ for $i \in [n]$ for $n$ so that the binomial coefficients corresponding to the size of the $d$-wise intersections are odd for $d<t$. The first open case is $(3,3)$-oddtown for which the above construction has size $\Theta(n^{1/2})$. If $\A =
\{ A_1, \ldots, A_m \}$ satisfies $(3,3)$-oddtown rules, then $\A' = \{ A_1 \cap A_i : 2 \leq i \leq m \}$ satisfies oddtown rules and thus $|\A| \leq n+1$. This is unlikely to be best possible as this holds replacing $A_1$ with \textit{any} $A_j \in \A$.

When $t<k$ and $2t-2>k$, we are able to show that $m = O(n^{1/(k-t+1)})$ (see Section \ref{sec:proofofmainthm2}), and conjecture that a stronger bound holds:   

\begin{conj}\label{conj:larget}
Let $t,k$ be integers with $t \geq 2$ and $2t-2 > k$. If $(\mathcal{A}_1,\mathcal{A}_2,\dots,\mathcal{A}_k)$ are set families of an $n$ element set with $\mathcal{A}_{j} = \{A_{j,i} : 1 \leq i \leq m\}$ where $|\bigcap_{j = 1}^k A_{j,i_j}|$ is even if and only if at least $t$ of the $i_j$ are distinct, then $m = O(n^{ 1/ \lfloor k/2 \rfloor})$.
\end{conj}

The bound $m = O(n^{ 1/ \lfloor k/2 \rfloor})$ comes from a connection to a hypergraph covering problem (see Section \ref{subsec:hgp}) and would imply that Conjecture \ref{conj:larget} if true is tight. 

Many results in extremal set theory regarding even/odd sized intersections have been extended to arbitrary primes $p$. Given a prime $p$ and a set $L \subset [0,p-1]$, $\A \subseteq 2^{[n]}$ is a \textit{$(p,L)$-intersecting} set system if for all $A \in \A$, $|A| \notin L \pmod{p}$, but for all distinct $A_1, A_2 \in \A$, $|A_1 \cap A_2| \in L \pmod{p}$. 
For more on $(p,L)$-intersecting set systems, see Deza-Frankl-Singhi \cite{DFS} as well as Frankl and Tokushige \cite{FT}. For $k \geq 3$, Szab{\' o} and Vu \cite{SV2} consider $k$-wise oddtown problems wherein the size of the sets are odd and the sizes of intersections of $k$ distinct sets are even. Grolmusz and Sudakov \cite{FS} study $k$-wise $(p,L)$-intersecting set systems where the size of distinct $k$-wise intersections are in $L \pmod{p}$. Similarly, we propose the following extension to Theorem \ref{thm:main1} for primes $p>2$. 

\begin{conj}\label{conj:main}
Let $p$ be prime and $(\mathcal{A}_1,\mathcal{A}_2,\dots,\mathcal{A}_k)$ be set families of an $n$ element set with $\mathcal{A}_{j} = \{A_{j,i} : 1 \leq i \leq m\}$ where
$|\bigcap_{j = 1}^k A_{j,i_j}|$ is nonzero modulo $p$ if and only if $i_1, \ldots, i_k \in [m]$ are all distinct. Then $m =  O(n^{(p-1) / \lfloor k/2 \rfloor})$.
\end{conj}

F{\" u}redi and Sudakov \cite[Lemma 3.1]{FS} used multilinear polynomials to prove if $\A= \{ A_1, \ldots, A_m\}$ and $\B=\{B_1, \ldots, B_m\}$ are families of subsets of an $n$ element set where $|A_i \cap B_j|$ is nonzero modulo $p$ if and only if $i\neq j$, then 
\begin{equation}\label{eq:pboundssetpairs}
m \leq \sum_{i=0}^{p-1} \binom{n}{i}.     
\end{equation}

It follows that \eqref{eq:pboundssetpairs} is asymptotically best possible by taking $\A = \binom{[n]}{p-1}$ and $\B$ to be the corresponding complements. This verifies Conjecture \ref{conj:main} when $k=2$.

\textbf{Structure and Organization:}
Throughout the paper, we let $[n]:= \{1,2, \ldots, n \}$ and $\binom{[n]}{k}$ denote the complete $k$-uniform hypergraph, or $k$-graph. Let $K_{n:k}$ denote the \textit{Kneser graph} with $V(K_{n:k}) = \binom{[n]}{k}$ and $(A,B) \in E(K_{n:k})$ if and only if $A \cap B = \emptyset$. Given subsets $X_1, \ldots, X_k$ of $[n]$, let $\prod_{i=1}^k X_i$ denote the complete $k$-partite $k$-graph with edges consisting of exactly one vertex from each part $X_1, \ldots, X_k$. For $k=2$, we refer to complete bipartite graphs as \textit{bicliques}. Given a subset $A \subseteq [n]$, let $v_X \in \mathbb{F}_2^n$ denote its characteristic vector and $(v_X)_k$ denote its $k$th component. For a family $\A= \{A_1, \ldots , A_m\} \subseteq 2^{[n]}$, let $\A^c= \{A_1^c, \ldots, A_m^c \} \subseteq 2^{[n]}$ be the corresponding complements. We denote the adjacency matrix of a graph $G$ by $A(G)$. For functions $f,g :\mathbb N\rightarrow \mathbb R^+$, we write $f\sim g$ if $\lim_{n \rightarrow\infty} f(n)/g(n) = 1$, $f = o(g)$ if $\lim_{n \rightarrow \infty} f(n)/g(n) = 0$, and $f = O(g)$ and $g = \Omega(f)$ if there exists $c > 0$ such that $f(n)\leq cg(n)$ for all $n \in \mathbb N$. If $f = O(g)$ and $g = O(f)$, we write $f =\Theta(g)$.

The paper is organized as follows. In Section \ref{subsec:hgp}, we introduce a hypergraph covering problem which we will use to show Theorem \ref{thm:main1} and Theorem \ref{thm:main2} are best possible. We prove Theorem \ref{thm:main1} and Theorem \ref{thm:main2} in Section \ref{sec:proofofmainthm1} and Section \ref{sec:proofofmainthm2} respectively.

\section{Hypergraph Covering Problems}\label{subsec:hgp} 

In this section, we relate the problem in Theorem \ref{thm:main1} and Theorem \ref{thm:main2} to that of a covering problem of a particular hypergraph. We first define the corresponding collection of set families.

\begin{defn}\label{defn:bsktupmodp}
Let $\A_j = \{A_{j,i} : 1 \leq i \leq m  \}$ for $j \in [k]$. Then $(\A_1, \A_2, \ldots,  \A_k)$ forms a \textit{Bollob{\'a}s set $(k,t)$-tuple modulo $2$} if \[ | \bigcap\limits_{i=1}^k A_{j,i_j}| = 0 \pmod{2} \iff |\{  i_1, \ldots, i_k\}| <t .  \] 
Let $b_{k,t}(n)$ be the size of the largest Bollob{\'a}s set $(k,t)$-tuple modulo $2$ with ground set $[n]$.
\end{defn}

In the setting where $k=t=2$, we refer to a Bollob{\'a}s set $(2,2)$-tuple modulo $2$ as a \textit{Bollob{\'a}s set pair modulo $2$}. Using the language of Definition \ref{defn:bsktupmodp}, Theorem \ref{thm:main1} and Theorem \ref{thm:main2} say $b_{k,k}(n) = \Theta(n^{1/\lfloor k/2 \rfloor})$ and $b_{k,t}(n) = \Theta(n^{1/(t-1)})$ respectively as one may add an auxiliary element to each set in each family to interchange the parity of the sets and their corresponding intersections. In order to determine of the asymptotics of $b_{k,t}(n)$, we consider the following related hypergraph covering problem:

\begin{defn}\label{defn:mod2cover}
Let $H \subseteq \binom{[n]}{k}$. A \textit{modulo $2$} cover of $H$ is a collection of complete $k$-partite $k$-graphs which cover each edge of $H$ an odd number of times and each non-edge an even number of times. We denote the minimum size of a modulo $2$ cover of $H$ by $f'_k(H)$. 
\end{defn}
 
Setting $X_{i} = \{x_{i,j} : j \in [n]\}$ for $i \in [k]$, let $H_{k,t}(n)$ be the subhypergraph of $\prod_{i \in [k]} X_i$ consisting of hyperedges $\{x_{1,i_1},x_{2,i_2},\dots,x_{k,i_k}\}$ such that at least $t$ of the indices $i_1,i_2,\dots,i_k$ are distinct. Then there is a one to one correspondence between a modulo $2$ cover of $H_{k,t}(n)$ with $m$ complete $k$-partite $k$-graphs and a Bollob{\'a}s set $(k,t)$-tuple modulo $2$ consisting of subsets of $[m]$. Hence, 
\begin{equation}\label{eq:betamugen}
f_{k,t}(n):= f'_k(H_{k,t}(n)) = \min\{m : b_{k,t}(m) \geq n\}.
\end{equation}

The problem in Definition \ref{defn:mod2cover} is closely related to the hypergraph extension of the Graham-Pollak \cite{GP} problem. Given a hypergraph $H \subseteq \binom{[n]}{k}$, let $f_k(H)$ denote minimum number of complete $k$-partite $k$-graphs needed to cover each edge of $H$ exactly once and let $f_k(n) = f_k(\binom{[n]}{k})$. 

Graham and Pollak \cite{GP} proved for all $n \geq 1$, $f_2(n)=n-1$. Alon \cite{ALON2} proved that $f_3(n) =n-2$ for $n\geq 2$ and proved that $f_k(n)= \Theta(n^{\lfloor k/2 \rfloor})$. Relating the problem to that of a minimal biclique cover of the Kneser graph $K_{n:k}$,  Cioab\u{a} et al. \cite{CKV} improved the bounds on $f_k(n)$ to:

\begin{equation}\label{eq:hypgpbounds}
 \frac{2 \binom{n-1}{k}}{\binom{2k}{k}} \leq f_{2k}(n) \leq \binom{n-k}{k}.   
\end{equation}

Currently Leader, Mili\'{c}evi\'{c}, and Tan \cite{LMT} have the best general bounds. A key ingredient in our proof of Theorem \ref{thm:main1} is to similarly relate this hypergraph covering problem to a graph covering problem as in Cioab\u{a} et al. \cite{CKV}. To this end, define the \textit{ordered} Kneser graph $OK_{n:k}$ as follows:  
\begin{eqnarray*}
V(OK_{n:k}) &=& \{ ( i_1, i_2, \ldots, i_k) : \{i_1, \ldots, i_k\} \in \binom{[n]}{k} \} \\ 
E(OK_{n:k}) &=& \{ ((i_1, \ldots, i_k), (j_1, \ldots, j_k)) : \{i_1, \ldots, i_k\} \cap \{j_1, \ldots, j_k\} = \emptyset  \}. 
\end{eqnarray*}

The following lemma relates modulo $2$ covers of $H_{2k,2k}(n)$ and biclique covers of $OK_{n:k}$. 

\begin{lemma}\label{lemma:reductiontokneser}
Let $\{ \mathcal{C}_i: i \in [t] \}$ be a modulo $2$ cover of $H_{2k,2k}(n)$. Then, there exists a modulo $2$ biclique cover $\{ \mathcal{C}_i': i \in [t] \}$ of $OK_{n:k}$ and in particular $f_{2k,2k}(n) \geq f'_2(OK_{n:k})$.
\end{lemma}

\begin{proof}
For each complete $2k$-partite $2k$-graph $\mathcal{C}_i := \prod_{j=1}^{2k} X_{i,j}$ in a modulo $2$ cover of $H_{2k,2k}(n)$, define the complete bipartite graph $\mathcal{C}_i'$ with parts 
$\prod_{j=1}^k X_{i,j} \cap V(OK_{n:k})$, and  $\prod_{j=k+1}^{2k} X_{i,j} \cap V(OK_{n:k})$.

Observe that $\{\mathcal{C}_i'\}$ is a modulo $2$ biclique cover of $OK_{n:k}$. 
\end{proof} 

Our next lemma shows a connection between the $k$-graphs $H_{k,k}(n)$ and $\binom{[n]}{k}$.

\begin{lemma}\label{lemma:gpupperbound}
Let $k \geq 2$. Then $f_k(H_{k,k}(n)) \leq k! f_k(n)$. 
\end{lemma}

\begin{proof}
Let $m= f_k(n)$ and $\{\mathcal{C}_i: i \in [m] \}$ be such a minimal $k$-partite, $k$-uniform cover of $\binom{[n]}{k}$. Then, given $\mathcal{C}_i= \prod_{i \in [k]} X_i$ and a bijection $\pi: [k] \to [k]$, define \[ \mathcal{C}_{i, \pi}= \prod_{i \in [k]} X_{\pi(i)} \subset H_{k,k}(n). \]
It then follows that $\{ \mathcal{C}_{i, \pi} \}$ forms a decomposition of $H_{k,k}(n)$ and thus implies the bound. 
\end{proof}

\section{Proof of Theorem \ref{thm:main1}}\label{sec:proofofmainthm1}

\subsection{The case $k=2$} In the setting where $t=k=2$, it is straightforward to show:

\begin{prop}\label{prop:pairsmod2} For all $n \geq 1$,

\begin{equation*}\label{eq:bspmod2}
f_{2,2}(n) = \begin{cases} 
      n & n \text{ is odd} \\
      n-1 & n \text{ is even} \\
     \end{cases} 
\end{equation*}

\end{prop}

\begin{proof}
As a result of \eqref{eq:betamugen}, it suffices to show $b_{2,2}(n)=n+1$ for $n$ even and $b_{2,2}(n)=n$ for $n$ odd.

For $n$ even, and $i \in [n]$ consider $A_i= \{i\}$ with $B_i = \{i\}^c$ and $B_{n+1} = A_{n+1} = [n]$. This is a Bollob{\'a}s set pair modulo $2$ of size $n+1$ and thus $b_{2,2}(n) = n+1$ for even $n$ by the bounds in \eqref{eq:pboundssetpairs}. 

For $n$ odd, let $m=b_{2,2}(n)$ and $(\A, \B)$ be a Bollob{\'a}s set pair modulo $2$ of size $m$. If $\{v_{A_i}: A_i \in \A \}$ is linearly independent in $\mathbb{F}_2^n$, then $m \leq n$. Thus, we may assume that there exists a non-trivial solution to $\sum_{i=1}^{m} \epsilon_i v_{A_i} = 0$ in $\mathbb{F}_2^n$. For all $j \in [m]$, it follows that \[ 0 = \bigg\langle \sum_{i=1}^{m} \epsilon_i v_{A_i}, v_{B_j} \bigg\rangle = \sum_{ \epsilon_i = 1; i \neq j} \langle v_{A_i}, v_{B_j} \rangle = \sum_{ \epsilon_i = 1; i \neq j} |A_i \cap B_j| = | \{ i: \epsilon_i =1, i \neq j \}| \pmod{2}. \] 
Since this is true for all $j \in [m]$, the existence of one non-zero $\epsilon_i$ implies they are all equal to one. In particular, $(m-1)$ is even and hence $m \in \{n,n+1\}$ is odd and thus $m=n$. 
\end{proof}

\subsection{The case $k \geq 3$} In this section, we will prove that
\begin{eqnarray}
\Bigl( \frac{1}{2} + o(1) \Bigr) \binom{n}{k} &\leq&    f_{2k,2k}(n) \leq (2k)! \binom{n}{k} \label{eq:thm2trueboundseven}  \\ 
\Bigl( \frac{1}{2} + o(1) \Bigr) \binom{n}{k-1}  &\leq&    f_{2k-1,2k-1}(n) \leq (2k-1)! \binom{n}{k-1} \label{eq:thm2trueboundsodd}
\end{eqnarray}
which imply Theorem \ref{thm:main1} by \eqref{eq:betamugen} and that the bound in the theorem is the correct order of magnitude. Using Lemma \ref{lemma:gpupperbound} with the upper bound in \eqref{eq:hypgpbounds}, we recover the upper bounds in \eqref{eq:thm2trueboundseven} and \eqref{eq:thm2trueboundsodd}. 

Next, by considering the link hypergraph of a vertex, it follows that $f_{k,k}(n) \geq f_{k-1,k-1}(n-1)$ and hence the lower bound in \eqref{eq:thm2trueboundsodd} follows by the corresponding lower bound in \eqref{eq:thm2trueboundseven}. As such, it suffices to prove the lower bound in \eqref{eq:thm2trueboundseven}.

To this end, let $\{ \mathcal{C}_i : i \in [t] \}$ be a modulo $2$ cover of $H_{2k,2k}(n)$ with $t= f_{2k,2k}(n)$. Lemma \ref{lemma:reductiontokneser} then yields that $t \geq f'_2(OK_{n:k})$ and as such we will prove a lower bound on $f'_2(OK_{n:k})$.   

As the adjacency matrix of each biclique has rank at most $2$, the subadditivity of matrix rank and taking the minor of $OK_{n:k}$ which corresponds to the Kneser graph $K_{n:k}$ give
\begin{equation}\label{eq:matrixsubadd}
f'_2(OK_{n:k}) \geq \frac{1}{2} \text{rank}(A(OK_{n:k})) \geq  \frac{1}{2} \text{rank}(A(K_{n:k}))    
\end{equation}
where the rank of the matrices is over $\mathbb{F}_2$. Let $M_{n,k,l} \in \{0,1\}^{ \binom{n}{k} \times \binom{n}{l}}$ where the $(K,L)$ entry for $K \in \binom{[n]}{k} $ and $L \in \binom{[n]}{l}$ is the indicator of $K \subseteq L$ and note that $M_{n,k,n-k} = A(K_{n:k})$.

Wilson \cite[Theorem 1]{W} determined the rank for matrices $M_{n,k,l}$ over $\mathbb{F}_p$ for $k \leq \min \{l,n-l\}$: 
\begin{equation}\label{eq:prankkneser}
\text{rank}(M_{n,k,l}) = \sum_{i: \; p \nmid \binom{n-k-i}{k-i}} \binom{n}{i} - \binom{n}{i-1}.     
\end{equation}

Since the corresponding binomial coefficients when $i \in [k-3,k]$ are odd when $n-2k=24\pmod{36}$, applying \eqref{eq:prankkneser} when $p=2$ gives the following for the rank over $\mathbb{F}_2$:
\begin{equation}\label{eq:rankkneser2}
\text{rank}(A(K_{n:k})) \geq \binom{n}{k} - \binom{n}{k-4}.
\end{equation}

Hence, standard binomial estimates and using \eqref{eq:matrixsubadd} and \eqref{eq:rankkneser2} give \[ f_{2k,2k}(n) \geq f'_2(OK_{n:k}) \geq \Bigl( \frac{1}{2} + o(1)\Bigr) \binom{n}{k}.   \]
Taking the bounds from \eqref{eq:thm2trueboundseven} and \eqref{eq:thm2trueboundsodd} and noting in \eqref{eq:betamugen} that $f_{k,k}(n) = \min\{m : b_{k,k}(m) \geq n\}$, we recover the following bounds on $b_{k,k}(n)$: 
\begin{eqnarray*}
\Bigl( \frac{1}{2k} + o(1)  \Bigr) n^{\frac{1}{k}}  &\leq& b_{2k,2k}(n) \leq \Bigl(k2^{\frac{1}{k}}\Bigr)n^{\frac{1}{k}} \\ \Bigl( \frac{1}{2k^2}  + o(1)    \Bigr)n^{\frac{1}{k-1}}  &\leq& b_{2k-1,2k-1}(n) \leq \Bigl((k-1)2^{\frac{1}{k-1}} + o(1)\Bigr)n^{\frac{1}{k-1}}. 
\end{eqnarray*}

\section{Proof of Theorem \ref{thm:main2}} \label{sec:proofofmainthm2}  
In this section, we will prove that if $2t-2 \leq k$, then
\begin{equation}\label{eq:thm3truebounds}
f_{k,t}(n) \leq  \Bigl( \frac{t^k}{t!}+o(1)\Bigr)n^{t-1} \hspace{3mm} \text{and} \hspace{3mm}  b_{k,t}(n) \leq \Bigl(t-1+o(1)\Bigr)n^{\frac{1}{t-1}}     .   
\end{equation}
Since $f_{k,t}(n) = \min\{m : b_{k,t}(m) \geq n\}$, the upper bound on $f_{k,t}(n)$ in \eqref{eq:thm3truebounds} implies a lower bound on $b_{k,t}(n)$ which together with the upper bound on $b_{k,t}(n)$ in \eqref{eq:thm3truebounds} proves Theorem \ref{thm:main2} as they imply
\[ \Bigl( \Bigl( \frac{t!}{t^k}\Bigr)^{\frac{1}{t-1}} + o(1) \Bigr) n^{\frac{1}{t-1}} \leq b_{k,t}(n) \leq \Bigl(t-1+o(1)\Bigr)n^{\frac{1}{t-1}} .      \]

\subsection{Upper bound on $f_{k,t}(n)$} 
Let $\phi: H_{k,t}(n) \to \{ \pi : |\pi| \geq t  \}$ be the correspondence of an edge of $H_{k,t}(n)$ to a set partition of $[k]$ where each part consists of the coordinates with equal indices. For a set partition $\pi$ of $[k]$, define $H_{k,t}^\pi(n) = \{ e \in H_{k,t}(n) : \phi(e) = \pi   \}$ so that 
\begin{equation}\label{eq:hypdecompsetpartition}
 H_{k,t}(n) = \bigsqcup\limits_{ |\pi| \geq t} H_{k,t}^\pi(n).   
\end{equation}
It follows that $f'_k(H_{k,t}^\pi(n)) \leq (n)_{(|\pi|-1)} = n(n-1) \cdots (n- |\pi|-2)$ by taking complete $k$-partite $k$-graphs corresponding to choosing a distinct singleton for the parts corresponding to $|\pi|-1$ of the parts and taking the remaining parts to be all of the remaining $(n-|\pi|+1)   $ vertices. Moreover, adding the complete $k$-partite $k$-graph $[n]^k$ to a minimal cover, it follows that 
\begin{equation}\label{eq:addcomp}
f_{k,t}(n) \leq 1 + f'_k([n]^{k} \setminus H_{k,t}(n)).     
\end{equation}

Since $f'_k(\cdot)$ is subadditive with respect to disjoint unions, using \eqref{eq:hypdecompsetpartition} gives
\begin{equation}\label{eq:upperboundsstirling}
f'_k([n]^{k} \setminus H_{k,t}(n)) \leq \sum_{|\pi| \leq t-1} f'_k(H_{k,t}^\pi(n)) \\ \leq \sum_{|\pi| \leq t-1} (n)_{(|\pi|-1)} = (S(k,t) + o(1))n^{t-1}    
\end{equation}
where $S(k,t)$ is the Stirling number of the second kind. We then recover the upper bound on $f_{k,t}(n)$ in \eqref{eq:thm3truebounds} by using the bound $S(k,t) \leq t^k / t!$ together with \eqref{eq:addcomp} and \eqref{eq:upperboundsstirling}.

\subsection{Upper bound on $b_{k,t}(n)$} 

We first prove the following lemma which relates $b_{k,t}(n)$ to $b_{2,2}(n)$. 

\begin{lemma}\label{lemma:middleupperbound}
Given $n \geq 1$, let $\alpha=\alpha(k,t) = 2t-k+2$. Then letting $m = b_{k,t}(n)$,
\[          
b_{2,2}(n) \geq 
\begin{cases}
{m \choose t-1} \hspace{15pt} \text{if} \hspace{5pt}  2t-2 \leq k \vspace{1mm} \\
{m -\alpha \choose k-t+1}  \hspace{15pt} \text{if} \hspace{5pt}  2t-2 > k \\
\end{cases}
\]
In particular, both quantities are less than or equal to $n+1$ by \eqref{eq:pboundssetpairs}. 
\end{lemma} 

\begin{proof}
Let $m= b_{k,t}(n)$ for $2t-2 \leq k$ and let $(\A_1, \ldots, \A_k)$ be a Bollob{\'a}s set $(k,t)$-tuple modulo $2$ of size $m$ where $\A_i = \{ A_{j,i} \subseteq [n] : i \in [m] \}$ for all $j \in [k]$. For each $I = \{i_1 < \cdots < i_{t-1}\} \in \binom{[m]}{t-1}$, define the families $\A := \{A_I\}$ and $\B:=\{ B_I\}$ where: 
\begin{eqnarray*}
A_I &:=& A_{1,i_1} \cap \cdots \cap A_{t-1, i_{t-1}} \\
B_I &:=& A_{t,i_1} \cap \cdots \cap A_{2t-2, i_{t-1}}
\end{eqnarray*}
Note that $(\A,\B)$ is a Bollob{\'a}s set pair modulo $2$ and hence the result follows.  

Similarly, for $2t-2 > k$, let $m= b_{k,t}(n)$ and let $(\A_1, \ldots, \A_k)$ be a Bollob{\'a}s set $(k,t)$-tuple modulo $2$ of size $m$ where $\A_i:= \{ A_{j,i} \subseteq [n]: i \in [m] \}$ for all $j \in [k]$. Consider the index set \[ \F := \binom{[\alpha +1, m]}{k-t+1}. \]
For $I = \{i_1 < \cdots < i_{k-t+1}\} \in \F$, define the families $\A := \{A_I\}$ and $\B:=\{ B_I\}$ where 
\begin{eqnarray*}
A_I &:=& A_{1,1}  \cap \cdots \cap A_{\alpha, \alpha} \cap A_{\alpha+1, i_1} \cap \cdots \cap A_{t-1,i_{k-t+1}} \\ 
B_I &:=& A_{k-t+2,i_1} \cap \cdots \cap A_{k, i_{t-1}}. 
\end{eqnarray*} 
Note that $(\A,\B)$ is a Bollob{\'a}s set pair modulo $2$ and hence the result follows. 
\end{proof}

We are now able to prove the upper bound on $b_{k,t}(n)$ in \eqref{eq:thm3truebounds} when $2t-2 \leq k$. Letting $m= b_{k,t}(n)$, Lemma \ref{lemma:middleupperbound} yields that ${ m \choose t-1} \leq n+1$. Using the bound ${ m \choose t-1} \geq (m/(t-1))^{t-1}$, it follows that \[ m \leq \Bigl(t-1+o(1)\Bigr)n^{\frac{1}{t-1}}.   \] 

\section{Concluding Remarks}

$\bullet$ Conjecture \ref{conj:main} involves extending Theorem \ref{thm:main1} to primes $p>2$. For $k \geq 3$, and $p=2$, we used a hypergraph extension of the Graham-Pollak theorem \cite{GP} to show that Theorem \ref{thm:main1} is best possible. For $p>2$, a promising approach seems to be a hypergraph extension of Alon's \cite{ALON1} result for bipartite coverings of order $\ell$ of the complete graph. To this end, for an integer $\ell \geq 1$, we say complete $k$-partite $k$-graphs $H_1, \ldots, H_m$ form a \textit{$k$-partite covering of order $\ell$} of $\binom{[n]}{k}$ if each edge of $\binom{[n]}{k}$ lies in at least one and at most $\ell$ of the $H_1, \ldots, H_m$. The following conjecture would show that Conjecture \ref{conj:main} is best possible: 

\begin{conj}
Let $k \geq 2$ and $\ell \geq 1$. Let $H_1, \ldots, H_m$ be complete $k$-partite $k$-graphs which form a $k$-partite covering of order $\ell$ of $\binom{[n]}{k}$ of minimum size. Then $m = \Theta(n^{  \lfloor k/2 \rfloor / \ell})$.   
\end{conj}

$\bullet$ In this paper, we used a result of Wilson \cite{W} on the rank over $\mathbb{F}_2$ of an incidence matrix to prove Theorem \ref{thm:main1}. To apply this same argument for general primes $p$ in Conjecture \ref{conj:main}, one would need to replace the entries which correspond to incidences with \textit{any} nonzero entry of $\mathbb{F}_p$. For $p=2$, it follows that this corresponds to the classical $\{0,1\}$ incidence matrix. A positive solution to the following problem would imply Conjecture \ref{conj:main}: 

\begin{problem}
Let $k \geq 2$ and $p$ be prime. Let $M^\star_{n,k,n-k}$ be any $\binom{n}{k} \times \binom{n}{n-k}$ matrix with entries in $\mathbb{F}_p$ where the $(K,L)$ entry for $K \in \binom{[n]}{k} $ and $L \in \binom{[n]}{n-k}$ is a nonzero element of $\mathbb{F}_p$ if $K \subseteq L$ and is zero otherwise. Prove that $\text{rank}(M^\star_{n,k,n-k}) = \Omega( n^{k/(p-1)})$ where the rank is over $\mathbb{F}_p$. 
\end{problem}

$\bullet$ In this paper, we prove bounds on $f_{k,t}(n)$ to prove Theorem \ref{thm:main1} and \ref{thm:main2}. These bounds imply
\begin{eqnarray}
\frac{1}{2 (\lfloor k/2 \rfloor!) } &\leq& \lim_{n \to \infty} \frac{f_{k,k}(n)}{n^{\lfloor k/2 \rfloor}} \leq (k)_{\lceil k/2 \rceil} \label{eq:limitfkk}  \\ 
\frac{1}{(t-1)^{t-1}} &\leq& \lim_{n \to \infty} \frac{f_{k,t}(n)}{n^{t-1}} \leq \frac{t^k}{k!} \label{eq:limitsmallt}
\end{eqnarray}
where we let $(k)_r=k(k-1)\cdots(k-r+1)$ denote the falling factorial and assume that $2t-2 \leq k$ in \eqref{eq:limitsmallt}. We believe that the upper bound in \eqref{eq:limitfkk} might be closer to the truth. It is likely that determining the limiting value is challenging as the problem is related to the hypergraph extension of the Graham-Pollak problem for which asymptotically tight bounds are not known. We also believe that the upper bound in \eqref{eq:limitsmallt} is closer to the truth.

$\bullet$ In this paper, we determined the order of magnitude for $b_{k,t}(n)$ in the case when $k=t$ and in the case when $2t-2 \leq k$. For small values of $t$ and $k$, more delicate arguments give better bounds. First, when $t=2$ and $k \geq2$ it is straightforward to see $ n-1 \leq b_{k,2}(n) \leq b_{2,2}(n)$ for all $n\geq 1$.

Although the asymptotics of the largest Bollob{\'a}s $(3,3)$-tuple\footnote{the sizes of the intersections are zero and nonzero accordingly as opposed to zero and nonzero modulo $2$} and the largest Bollob{\'a}s $(4,3)$-tuple\footnotemark[1] are wide open (see \cite{OV}), in the modulo $2$ setting we are able to show: 
\begin{eqnarray}
\frac{n-1}{3} &\leq& b_{3,3}(n) \leq n+2 \label{eq:smallvalues33}\\
\frac{\sqrt{3n+1}-2}{3} &\leq& b_{4,3}(n) \leq \sqrt{2n+2} + 4 \label{eq:smallvalues43}.
\end{eqnarray}
We start by sketching the proofs of the lower bounds in \eqref{eq:smallvalues33} and \eqref{eq:smallvalues43} by proving upper bounds on $f_{3,3}(n)$ and $f_{4,3}(n)$ and using \eqref{eq:betamugen}. For $f_{3,3}(n)$, a modulo $2$ cover of $H_{3,3}(n)$ is given by $\{i\} \times \{i\} \times [n]$, $\{i\} \times [n] \times \{i\}$, and $[n] \times \{i\} \times \{i\}$ for any index $i \in [n]$ together with $[n]^{3}$. Hence $f_{3,3}(n) \leq 3n+1$. For $f_{4,3}(n)$, a similar construction gives $f_{4,3}(n) \leq f_{4,2}(n) + 3n(n-1) + 6n$. It is straightforward to see $f_{4,2}(n) \leq n+1$ and hence $f_{4,3}(n) \leq n+1 + 3n(n-1) + 6n = 3n^2+4n+1$.  

We next sketch the proofs of the upper bounds in \eqref{eq:smallvalues33} and \eqref{eq:smallvalues43}. The upper bound on $b_{4,3}(n)$ follows from Lemma \ref{lemma:middleupperbound}. For the upper bound on $b_{3,3}(n)$, consider any Bollob{\'a}s set triple modulo $2$ denoted by $(\A_1, \A_2, \A_3)$ where $\A_j:= \{ A_{j,i} \subseteq [n] : i \in [m]   \}$ and define $ \F_1 = \{ A_{1,1} \cap A_{2,i} : 2 \leq i \leq m\}$ and $\F_2 = \{ A_{1,1} \cap A_{3,i} : 2 \leq i \leq m\}.$ It then follows that $(\F_1, \F_2)$ is a Bollob{\'a}s set pair modulo $2$. 

$\bullet$ Conjecture \ref{conj:larget} involves determining the order of magnitude of $b_{k,t}(n)$ when $2t-2>k$. Theorem \ref{thm:main1} verifies Conjecture \ref{conj:larget} when $t=k$. When $(k,t)= (5,4)$, we can use a similar argument as in the upper bound in \eqref{eq:smallvalues33} to show that $b_{5,4}(n)-1 \leq b_{4,3}(n)$.  
Using Lemma \ref{lemma:middleupperbound} and the upper bound in \eqref{eq:smallvalues43}, it then follows that $b_{5,4}(n)=\Theta(n^{1/2})$. This verifies Conjecture \ref{conj:larget} in the case where $(k,t)= (5,4)$. In general, Lemma \ref{lemma:middleupperbound} gives $b_{k,t}(n) = O(n^{1/(k-t+1)})$ and a more delicate argument similar to Lemma \ref{lemma:gpupperbound} together with \eqref{eq:betamugen} gives a construction of size $ \Theta(n^{1/ \lfloor k/2 \rfloor})$. The first open case is $(k,t) = (6,5)$ for which $b_{6,5}(n) = \Omega(n^{1/3})$ and $b_{6,5}(n) = O(n^{1/2})$.

\end{document}